\documentclass[leqno,12pt]{amsart} 
\usepackage{amssymb} 
 
\theoremstyle{plain} 
\newtheorem{thm}{\indent Theorem}[section] 
\newtheorem{lemma}[thm]{\indent Lemma}

\newtheorem{proposition}[thm]{\indent Proposition}

\theoremstyle{definition} 
\newtheorem{defn}[thm]{\indent Definition}
\newtheorem{example}[thm]{\indent Example}

\theoremstyle{remark}

\newcommand{\bR}{{\mathbb R}}

\usepackage{graphicx}

\begin{document}

\title[Calibrated submanifolds in neck manifolds]{Calibrated submanifolds in neck manifolds} 

\author[H. Nakahara]{Hiroshi Nakahara} 
\subjclass[2010]{53C38.
}

\address{
Department of Mathematics \endgraf 
Tokyo Institute of Technology \endgraf
 2-21-1  O-okayama, Meguro, Tokyo\endgraf
Japan
}
\email{12d00031@math.titech.ac.jp}

\maketitle

\begin{abstract}
We find calibrated submanifolds in neck manifolds. Particularly, we obtain a calibrated submanifold in the Lagrangian self-expander constructed by Joyce, Lee and Tsui. 
\end{abstract}
\section{\rm{Introduction}}
Minimal or volume-minimizing surfaces have been investigated since Lagrange considered the variational problem of finding the surface $z=z(x, y)$ of least area stretched across a given closed contour in 1762, and Harvey and Lawson invented the theory of calibrated geometry in \cite{HL} by which we can find volume-minimizing surfaces in Riemannian manifolds. For example, since both special Lagrangian submanifolds and a pair of oriented $m$-planes which satisfies the angle criterion, see also \cite{La}, are calibrated, they are volume-minimizing. In recent years the special Lagrangian submanifolds in Calabi-Yau $n$-folds has been extensively studied. For instance, it is a key ingredient in the Thomas-Yau Conjecture. It is well-known that Lawlor necks are explicit examples of special Lagrangian submanifolds in the complex Euclidean space. However, although Lawlor necks are very important in special Lagrangian geometry, little attention has been paid to the submanifolds inside. In this paper we consider neck-shaped manifolds $M\times N,$ to be explained below. In the neck manifolds, we get calibrations and the calibrated submanifolds. Particularly, we will see that the hypersurfaces $\{y=0\}$ in Lawlor necks or the Lagrangian self-expanders constructed by Joyce, Lee and Tsui in \cite[Theorem C]{JLT} are calibrated. 
The following Definition \ref{d} and Proposition \ref{p} are extracts from Joyce \cite[Chapter 4]{J}.

\begin{defn}\label{d}
Let $(M,g)$ be a riemannian manifold. An oriented tangent $k$-plane $V$ on $M$ is a vector subspace $V$ of some tangent space $T_x M$ to $M$ with $\dim V=k,$ equipped with an orientation. If $V$ is an oriented tangent $k$-plane on $M$ then $g|_V$ is a Euclidean metric on $V,$ so combining $g|_V$ with the orientation on $V$ gives a natural volume form ${\rm vol}_{V}$ on $V,$ which is a $k$-form on $V.$

Now let $\varphi$ be a closed $k$-form on $M.$ We say that $\varphi$ is a calibration on $M$ if for every oriented $k$-plane $V$ on $M$ we have $\varphi |_V \leq {\rm vol}_V.$ Here $\varphi |_V=\alpha \cdot {\rm vol}_V$ for some $\alpha \in \bR ,$ and $\varphi |_V \leq {\rm vol}_V$ if $\alpha \leq 1.$ Let $N$ be an oriented submanifold of $M$ with dimension $k.$ Then each tangent space $T_x N $ for $x\in N$ is an oriented tangent $k$-plane. We say that $N$ is a calibrated submanifold or $\varphi$-submanifold if $\varphi|_{T_x N}=
{\rm vol}_{T_x N} $ for all $x\in N.$   
\end{defn} 
\begin{proposition}\label{p}
Let $(M,g)$ be a riemannian manifold, $\varphi$ a calibration on $M,$ and $N$ a compact $\varphi$-submanifold in $M.$ Then $N$ is volume-minimizing in its homology class. 
\end{proposition}
We omit the proof of Proposition \ref{p}. The reader can check it in \cite[Section 4]{J}. 

The following Theorem \ref{th} is our main result.
\begin{thm}\label{th}
Let $M$ be an oriented submanifold in $\bR^n .$ Let $(N,h(p))$ be riemannian manifolds for all $p\in M ,$ where $\{h(p)\}_p$ is a smooth family of riemannian metrics on $N.$
Let $f_1 \ldots ,f_n :N \to \bR _+$ be positive and smooth functions and 
$g(q)$ riemannian metrics on $M $ defined by $g(q)=(\sum_{j=1}^n f_j ^2 (q) \, dx_j ^2)|_M ,$ for any $q\in N .$ Suppose that there exists a point $q_0 \in N$ such that $\Pi_{k=1}^n f_j (q_0)=\min_{q\in N} (\Pi_{k=1}^n f_j (q)) .$ We write $\pi : M\times N \to M $ for the projection such that $\pi(p, q)= p.$ 
Then $\pi ^* {\rm vol} _{g(q_0)}$ is a calibration on the  riemannian manifold $(M\times N, g(q)+h(p))$ where $ {\rm vol} _{g(q_0)}$ is the volume form on $(M, g(q_0)),$
and $M\times \{q_0\}$ is the calibrated submanifold. 
\end{thm}
Particularly, we have the following example. 
\begin{example}
Let $a_1 ,\ldots  ,a_n >0$ and $\alpha \geq 0$ be constants.  Define riemannian metrics $g(s) $ on $\mathcal {S}^{n-1}=\{(x_1 ,\ldots ,x_n)\in \bR ^n ; \sum_{j=1}^n x_j ^2 =1\}$ by $g(s)=(\sum_{j=1}^n (1/a_j +s^s)dx_j ^2)|_{\mathcal{S}^{n-1}},$ for every $s\in \bR ,$ and riemannian metrics $h(x)$ on $\bR$ by 
$$h(x)=(1/a_1 +s^2 )\cdots (1/a_n +s^2)\sum_{j=1}^n \frac{x_j ^2}{1/a_j +s^2 }\,ds^2$$ for all $x=(x_1 ,\ldots ,x_n)\in \mathcal{S}^{n-1}.$ 
Then we can regard $(\mathcal{S} ^{n-1} \times \bR ,g(s)+h(x)) $ as the Lagrangian self-expander constructed by Joyce, Lee and Tsui in \cite{JLT}. It has been proved in \cite{N} that when $a_1 =\cdots =a_n,$ $\mathcal{S} ^{n-1} \times \{0\} $ is minimal in  $(\mathcal{S} ^{n-1} \times \bR ,g(s)+h(x)) $ and,
by Theorem \ref{th}, $\pi^* {\rm vol}_{g(0)}$ is a calibration on  $(\mathcal{S} ^{n-1} \times \bR ,g(s)+h(x)) $ and $\mathcal{S} ^{n-1} \times \{0\} $ is the calibrated submanifold. 
\end{example}
\subsection*{Acknowledgments}
The author wishes to express his thanks to his supervisor Akito Futaki for a great encouragement and several helpful comments.  
\section{\rm{Proof of Theorem \ref{th}}}
From Lemmas \ref{1} and \ref{2}, we can obtain Theorem \ref{th}. 
\begin{lemma}\label{1}
Fix $k= \dim M$ and $p \in M$ in the situation of Theorem \ref{th}. For any $X_1 ,\ldots ,X_k \in T_{p}M$ and every $q\in N$ we have $$|{\rm vol}_{g(q_0)}(X_1 ,\ldots ,X_k)| \leq |{\rm vol}_{g(q)}(X_1 ,\ldots ,X_k)|$$ where ${\rm vol}_{g(q)}$ is the volume form on the riemannian manifold $(M, g(q)).$
\end{lemma}
\begin{proof}
Set $X_i = (X_{i1},\ldots ,X_{in})\in \bR ^n,$ for $i= 1,\ldots ,k,$ and  
\begin{equation*}
A(q)=
\begin{pmatrix}
f_1  (q)X_{11} &\cdots & f_{n}(q)X_{1n}\\
\vdots &  &\vdots \\
f_1  (q)X_{k1} &\cdots & f_{n}(q)X_{kn}\\
\end{pmatrix}
,
\end{equation*}
for all $q\in N. $
Write
\begin{equation*}
B=
\begin{pmatrix}
X_{11} &\cdots &X_{1n}\\
\vdots &  &\vdots \\
X_{k1} &\cdots & X_{kn}\\
\end{pmatrix}
\end{equation*}
and $[n]=\{1,\ldots ,n\}.$ Note that if  $A\in M_{k, n}(\bR)$ is a $k\times n$ matrix such that $n\geq k$ and $S\subset [n]$ such that the number of the element in $S$ is $k$, i.e. $|S|=k$, then we write $A_S$ for the $k\times k$ matrix whose columns are the columns of $A$ at indices from $S,$ and if $B\in M_{n, k}(\bR)$ is a $n\times k $ matrix such that $n\geq k$ and $S\subset [n]$ such that $|S|=k$ then we write $B^S$ for the $k\times k$ matrix whose rows are the rows of $B$ at indices from $S.$
By the Cauchy-Binet formula, we obtain 
\begin{equation}\label{1}
\begin{split}
({\rm vol}_{g(q)}(X_1 ,\ldots ,X_k))^2 =&\det (g(q)(X_i ,X_j))_{i, j}\\
=& \det \left(\sum_{l=1}^n f_l ^2 (q)X_{il}X_{jl}\right)_{i,j}\\
=& \det A(q)\,{}^t\!A (q) \\
=& \sum_{S\subset [n], |S|=k} \det (A(q))_S  \cdot \det ({}^t\!A(q))^S \\
=& \sum_{S\subset [n], |S|=k} (\det (A(q))_S)^2\\
=& f_1 (q)^2 \cdots f_n (q)^2 \sum_{S\subset [n], |S|=k} (\det (B)_S)^2 \,\,.\\
\end{split}
\end{equation}
Thus, by the definition of $q_0$ and \eqref{1}, we get $$|{\rm vol}_{g(q_0)}(X_1 ,\ldots ,X_k)| \leq |{\rm vol}_{g(q)}(X_1 ,\ldots ,X_k)|.$$
This completes the proof.
\end{proof}
Next we consider the following Lemma \ref{2}.
\begin{lemma}\label{2}
Fix $(p,q)\in M\times N $ in the situation of Theorem \ref{th}. Let $V\subset T_{(p,\,q)}(M\times N)$ be an oriented tangent $k$-plane, i.e. a vector subspace of $T_{(p,\,q)}(M\times N)$ with $\dim V= k,$ $(v_1 ,\ldots ,v_k)$ a basis of $\,V$ with the positive orientation. Then 
$$|{ \rm vol }_{g(q)}(\pi _* v_1 \ldots ,\pi _* v_k)|\leq { \rm vol }_{V}( v_1 \ldots ,v_k) .$$ 
\end{lemma}
\begin{proof}
Set $\dim N=t.$
Let $(e_1 ,\ldots ,e_k )$ be an orthonormal basis of $T_{p} M$ and
 
$(e_1 \ldots ,e_k ,e_{k+1},\ldots ,e_{k+t})$ be an orthonormal basis of $T_{(p,\,q)} M\times N .$ 
Write $v_j = \sum _{l=1}^{k+t}a_j ^l e_{l}$ for $j=1,\ldots k ,$ where $a_{j}^l \in \bR .$ 
Write 
\begin{equation*}
C=
\begin{pmatrix}
a_1 ^1 &\cdots &a_{1}^{k+t}\\
\vdots & & \vdots \\
a_k ^1 &\cdots &a_{k}^{k+t}\\
\end{pmatrix}
\end{equation*}
and
\begin{equation*}
D=
\begin{pmatrix}
a_1 ^1 &\cdots &a_{1}^{k}\\
\vdots & & \vdots \\
a_k ^1 &\cdots &a_{k}^{k}\\
\end{pmatrix}
.
\end{equation*}
Note that if $A\in M_{k, k+t}(\bR)$ is a $k\times (k+t)$ matrix and $S\subset [k+t]$ such that the number of the element in $S$ is $k$, i.e. $|S|=k$, then we write $A_S$ for the $k\times k$ matrix whose columns are the columns of $A$ at indices from $S,$ and if $B\in M_{k+t, k}(\bR)$ is a $(k+t)\times k $ matrix and $S\subset [k+t]$ such that $|S|=k$ then we write $B^S$ for the $k\times k$ matrix whose rows are the rows of $B$ at indices from $S.$
Now, by the Cauchy-Binet formula, we have 
\begin{equation}\label{2}
\begin{split}
( { \rm vol }_{V}( v_1 \ldots ,v_k) )^2=&\det \,\left((g(q)+h(p))(v_i ,v_j)\right)_{i,\,j}\\
=&\det \,\left((g(q)+h(p))\left(\sum _{l=1}^{k+t}a_i ^l e_{l},\sum _{b=1}^{k+t}a_j ^b e_{b}\right)\right)_{i,\,j}\\
=&\det \, \left(\sum_{l,\,b =1}^{k+l}a_i ^l a_j ^b \delta_{l\, b}\right)_{i\, ,j}\\
=&\det \, \left(\sum_{l}^{k+t}a_i ^l a_j ^l \right)_{i\, ,j}\\
=&\det \, C \,{}^t\! C\\
=& \sum_{S\subset [k+t], |S|=k} \det C_S  \cdot \det \,({}^t\!C)^S \\
=& \sum_{S\subset [k+t], |S|=k} (\det C_S)^2\\
\end{split}
\end{equation}
and 
\begin{equation}\label{3}
\begin{split}
({ \rm vol }_{g(q)}(\pi _* v_1 \ldots ,\pi _* v_k))^2=&\det \,\left(g(q)(\pi_*v_i ,\pi_*v_j)\right)_{i,\,j}\\
=&\det \,\left(g(q)\left(\sum _{l=1}^{k}a_i ^l e_{l},\sum _{b=1}^{k}a_j ^b e_{b}\right)\right)_{i,\,j}\\
=&\det \, \left(\sum_{l,\,b =1}^{k}a_i ^l a_j ^b \delta_{l\, b}\right)_{i\, ,j}\\
=&\det \, \left(\sum_{l}^{k}a_i ^l a_j ^l \right)_{i\, ,j}\\
=&\det \, D \,{}^t\! D\\
=&(\det \, D )^2 .\\
\end{split}
\end{equation}
By \eqref{2} and \eqref{3}, it is clear that 
$$|{ \rm vol }_{g(q)}(\pi _* v_1 \ldots ,\pi _* v_k)|\leq { \rm vol }_{V}( v_1 \ldots ,v_k) .$$ 
This finishes the proof.
\end{proof}
We are now ready to prove Theorem \ref{th}. 

\textit{Proof of Theorem {\rm \ref{th}.}}
Let $V\subset T_{(p,\,q)}(M\times N)$ be an oriented tangent $k$-plane on $M\times N$ and $(v_1 ,\ldots ,v_k)$ a basis of $\,V$ with the positive orientation. Then, from Lemmas \ref{1} and \ref{2}, we have 
\begin{equation}
\begin{split}
\pi ^* {\rm vol}_{g(q_0)}(v_1 ,\ldots ,v_k)=&{\rm vol}_{g(q_0)}(\pi _* v_1 ,\ldots ,\pi _* v_k) \\
\leq & |{\rm vol}_{g(q)}(\pi _* v_1 ,\ldots ,\pi _* v_k) |\\
\leq &{\rm vol}_{V}(v_1 ,\ldots ,v_k).
\end{split}
\end{equation}
Therefore the closed $k$-form $\pi ^* {\rm vol}_{g(q_0)}$ is a calibration on $M\times N. $ Furthermore it is clear that $M\times \{0\}$ is a calibrated submanifold with respect to $\pi ^* {\rm vol}_{g(q_0)}.$ This completes the proof.
\qed
\section{\rm{Discussion}}
The author believes that the length-minimizing curve in the dumbbell surface with a neck in $\bR^3$ becomes a singular point by the mean curvature flow. In general, if we consider the mean curvature flow of some submanifold, we have to check the necks in the submanifold. 

Volume-minimizing submanifolds correspond to many inequalities and the author hopes that many applications of this paper will appear.   
%
%
%

\end{document}